\documentclass[12pt]{article}
\usepackage{e-jc}

\usepackage{amsmath}
\usepackage{amssymb}
\usepackage{amsthm}
\usepackage{graphicx}
\usepackage{color}
\usepackage[colorlinks]{hyperref}

\definecolor{darkred}{rgb}{0.5,0,0}
\definecolor{darkgreen}{rgb}{0,0.5,0}
\definecolor{darkblue}{rgb}{0,0,0.5}
\hypersetup{colorlinks,linkcolor=darkred,filecolor=darkgreen,urlcolor=darkred,citecolor=darkblue}

\newtheorem{theorem}{Theorem}
\newtheorem{lemma}[theorem]{Lemma}
\newtheorem{corollary}[theorem]{Corollary}
\newtheorem{definition}[theorem]{Definition}

\DeclareMathOperator{\cp}{cp}

\newcommand{\eref}[1]{Equation~(\ref{#1})}
\newcommand{\tref}[1]{Theorem~\ref{#1}}
\newcommand{\lref}[1]{Lemma~\ref{#1}} 
\newcommand{\cref}[1]{Corollary~\ref{#1}}
\newcommand{\dref}[1]{Definition~\ref{#1}} 
\newcommand{\fref}[1]{Figure~\ref{#1}}

\begin{document}%

\title{Counting Defective Parking Functions}%

\author{Peter J Cameron, Daniel Johannsen, \\ Thomas Prellberg, Pascal Schweitzer\\
\small School of Mathematical Sciences\\[-0.8ex]
\small Queen Mary, University of London\\[-0.8ex]
\small Mile End Road, London E1 4NS, U.K.\\[-0.8ex]
\small \texttt{p.j.cameron@qmul.ac.uk} \\[-0.8ex]
\small \texttt{tp@maths.qmul.ac.uk} \\
\small Max-Planck-Institut f\"ur Informatik\\[-0.8ex]
\small CAMPUS E1 4, 66123 Saarbrücken\\[-0.8ex]
\small \texttt{johannse@mpi-inf.mpg.de} \\[-0.8ex]
\small \texttt{pascal@mpi-inf.mpg.de}  
}%

\date{\dateline{Mar 3, 2008}{not yet}}
\maketitle

\begin{abstract}%
Suppose that $m$ drivers each choose a preferred parking space in a linear car park with $n$ spaces. Each driver goes to the chosen space and parks there if it is free, and otherwise takes the first available space with a larger number (if any). If all drivers park successfully, the sequence of choices is called a parking function. In general, if $k$ drivers fail to park, we have a \emph{defective parking function} of \emph{defect} $k$. Let $\cp(n,m,k)$ be the number of such functions.

In this paper, we establish a recurrence relation for the numbers $\cp(n,m,k)$, and express this as an equation for a three-variable generating function. We solve this equation using the kernel method, and extract the coefficients explicitly: it turns out that the cumulative totals are partial sums in Abel's binomial identity. Finally, we compute the asymptotics of $\cp(n,m,k)$. In particular, for the case $m=n$, if choices are made  independently at random, the limiting distribution of the defect (the number of drivers who fail to park), scaled by the square root of $n$, is the Rayleigh distribution. On the other hand, in the case $m=\omega(n)$, the probability that all spaces are occupied tends asymptotically to one.
\end{abstract}%

\section{Introduction}%

A car park consists of $n$ numbered spaces in a line. The drivers of $m$ cars have independently chosen their favorite parking spaces. Each driver arrives at the car park and proceeds to his chosen space, parking there if it is free. If the chosen space is occupied, the driver continues on towards the larger-numbered spaces and takes the first available space if any; if no  such space is available, the driver leaves the car park and goes home. What is the probability that everybody parks successfully? Equivalently, how many of the $n^m$ sequences of choices by the drivers lead to everyone parking? (Such a sequence is called a \emph{parking function}.)

This problem was first raised in the 1960s in connection with hashing \cite{konheim_weiss}. In the case of $m=n$, a short and elegant proof of the formula $(n+1)^{n-1}$ was given by Pollak~\cite{foata_riordan}. From these beginnings, a substantial theory of parking functions has been developed, with links to trees (as one would expect from the formula above) and priority queues \cite{gilbey_kalikow},  partitions~\cite{stanley}, and representation theory~\cite{haiman}. More recently, generalizations of parking functions have found application in areas like the modelling of percolation \cite{majumdar_dean}, the Abelian sandpile model \cite{sandpile04} and branching processes~\cite{DumitriuSpencerYan03}.

In this paper,  we are concerned with the probability that $k$ drivers  fail to park successfully. We call the corresponding assignments a \emph{defective parking function} of \emph{defect}~$k$. Suppose that $m$ cars attempt to park in a linear car park with $n$ spaces according to the above rules; let $\cp(n,m,k)$ be the number of choices which result in exactly $k$ drivers failing to park. 

The concept is related to that of $x$--parking functions introduced by Stanley and Pitman: For a tuple of integers $x = (x_1, \ldots, x_n)$ with $n\in\mathbb N$, an $x$--parking function is a sequence $(a_1,\ldots,a_n)$ whose ordered permutation $(b_1,\ldots,b_n)$ satisfies $b_i \leq x_1+\ldots+x_i$ for all $1\le i\le n$.

These generalized parking functions have been extensively studied. Stanley and Pitman related their number to the volume polynomials of certain types of polytopes and to that of plane partitions \cite{pitmanstanley98}. Later, Yan and Kung  investigated moments of sums of their numbers by using Gon{\v c}arov polynomials~\cite{kungyan2003a,kungyan2003b}.

Classical parking functions correspond to $(1,1,\hdots,1)$--parking functions while defective parking functions of defect $k$ correspond to $(n-(m-k)+1,1,\ldots,1,0,\ldots,0)$--parking functions which in addition are not $(n-(m-k),1,\ldots,1,0,\ldots,0)$--parking functions.  For tuples $x=(a,b,\hdots,b,c,0,\hdots,0)$, an explicit formula was derived in \cite{pitmanstanley98} using the connection to polytope volumes and was reproven by Yan with combinatorial means in \cite{yan00}.

We establish a new recurrence relation for the number of defective parking functions, allowing us to formulate an equation defining the corresponding three-variable generating function. Applying the kernel method \cite{Flajolet,Prodinger}, we solve this equation explicitly, and then extract the coefficients. We reobtain the fact that the cumulative totals turn out to be partial sums in \emph{Abel's binomial formula} \cite{abelbinom} as shown in~\cite{pitmanstanley98} and \cite{yan00} within the context of $x$--parking functions. In fact, the parking function approach may be used to prove special cases of this identity.

We then investigate the asymptotical behavior of defective parking for prominent cases. Spencer and Yan~\cite{SpencerYan} have studied asymptotics of parking functions with a defect equal to the difference between the number of cars and the number of spaces (the case in which all parking spaces end up being taken). We extend these results to the case of arbitrary defects. In particular, we include the case in which the number of cars is less than the number of spaces.

First, we show that, for fixed $k$ and $\ell$, the limit of $\cp(n,n+\ell,k)/\cp(n,n+\ell,0)$ exists, and compute its value. For example, the limiting value of $\cp(n,n,1)/\cp(n,n,0)$ is $2\mathrm{e}-3$. 

To survey the limiting shape of the distribution, we need appropriate  scaling, which turns out to be by the square root of $n$. We show that, if  $m=n+\lfloor y\sqrt{n}\rfloor$, then the limiting probability of at most  $\lfloor x\sqrt{n}\rfloor$ drivers failing to park is
\begin{equation*}%
\lim_{n\to\infty}\frac{1}{n^m}\cdot\sum_{k=0}^{\lfloor x\sqrt{n}\rfloor} 
\cp(n,m,k)
= \begin{cases}%
1-e^{-2x(x-y)} &  \text{if } x>y,\\
0 & \text{otherwise};\\
\end{cases}%
\end{equation*}%
a surprisingly simple result, given the complicated form of the exact formula.

For $y=0$ (that is, in the case of $m=n$), this limiting distribution is the  \emph{Rayleigh distribution} with parameter $1/2$. (This occurs as the  distribution of the length of a random vector in the plane whose coordinates are independent normal variables with standard deviation $1/2$. We do not know of a direct connection of this with our problem.)

We also investigate the limiting probability that all parking spaces are occupied. Obviously, for $m$ strictly smaller than $n$ this probability is zero. We show that for $m=\lfloor\lambda n\rfloor$ with fixed $\lambda\in\mathbb R^+$ and $k=m-n$,
\begin{equation*}%
\lim_{n\to\infty}\frac{\cp(n,m,k)}{n^m}=
\begin{cases}%
0&\text{if } \lambda\leq 1,\\
1-e^{-{\lambda}}\cdot\sum_{i\ge 1}
\frac{(\lambda i/ e^{\lambda})^{i-1}}{i!}
&\text{if } \lambda>1.
\end{cases}%
\end{equation*}%

An alternative interpretation of the above car park problem involves a variation on the coupon collector problem. In the original problem, there are $n$ distinct items. If a collector acquires random items, she will have approximately $n/e$ duplicates after collecting the first $n$ items and will need to collect about $n\cdot\log n$ items before she has a complete set. But suppose the items are of strictly decreasing value and she has the option of trading duplicate items; each item may be traded for any other one of lower value.  Ideally, she trades duplicates against the next most valuable item she does not yet possess.   Then, we show that she will receive only about $\sqrt{n}$ duplicates among the first $n$ items and a complete collection already with $n\cdot f(n)$ items for any function $f\colon\mathbb N\to\mathbb N$ for which $\lim_{n\to\infty}f(n)=\infty$ holds.

We conclude this introduction with Pollak's lovely proof, adapted to the general case: for $m$ cars and $n$ spaces (with $m\le n$), the number of ways in which every driver parks successfully is $(n+1-m)(n+1)^{m-1}$. To see this, consider a circular car park with $n+1$ spaces, for which the same rules apply.  Now everyone will park successfully and there will be $n+1-m$ empty spaces; such a choice will be a parking function (for the original problem), if and only if space number $n+1$ is empty. By symmetry, this will happen in a fraction $(n+1-m)/(n+1)$ of the total number $(n+1)^m$ of choices. We will see later that our argument reproduces this result as an essential step in the working-out of the kernel method.

\section{A Functional Equation}%
\label{sec:recursion}%

Let $\cp(n,m,k)$ be the number of assignments of $m$ drivers to a car park with $n$ spaces, that result in exactly $k$ drivers leaving in the end, where the parking strategy of the drivers is as described in the introduction above. There are $n^m$ such assignments.  Then $\cp(n,m,k)$ can be concisely expressed as the number of functions $f\colon \{1,\hdots, m\} \rightarrow \{1,\hdots,n\}$ for which for all $i\in \{1,\hdots,n\}$ the set $f^{-1}(\{n+1-i,\hdots,n\})$ has a size of at most $k+i$, and at least one of these sets has a size of exactly $k+i$. 

Values of $\cp(n,n,k)$, the case in which the number of drivers and the number of spaces coincide, can be found in Table~\ref{table_of_values}.

\begin{table}%
\thinlines%
\begin{scriptsize}%
\begin{center}%
\thinlines%
  \begin{tabular}{ r | r  r  r r r r r r r r}%
$n$ & $k =0$& 1 & 2 & 3 & 4 & 5 & 6 & 7 & 8 & 9\\
    \hline
    1&   1 \\
    2&   3 &  1 &   \\
    3&   16 & 10 &  1  & \\
    4&   125& 107 & 23 & 1 &   \\
    5&   1296 & 1346 & 436 & 46 &  1 \\
    6&   16807  &  19917 & 8402 & 1442& 87 & 1\\
    7&   262144 &  341986 & 173860  & 41070 & 4320 & 162 & 1\\
    8&   4782969 & 6713975 & 3924685 & 1166083 & 176843 & 12357 & 303& 1\\
    9&   100000000 &  148717762 &  96920092 &   34268902 &   6768184 &   710314 &  34660 &  574&    1\\
   10&   2357947691& 3674435393& 2612981360& 1059688652& 256059854& 36046214& 2743112&  96620& 1103& 1\\
  \end{tabular}%
\end{center}%
\end{scriptsize}%
\caption{\label{table_of_values}The table shows the numbers $\cp(n,n,k)$ for $n=1,\hdots,10$ and $k=0,\hdots,9$ which count all car parking assignments of $n$ cars to $n$ spaces, such that $k$ cars are not parked.}%
\end{table}%

We will now derive a recursion formula by transforming the parameters so that they are more suitable for our purpose. Let $r$ be the number of spaces that end up unoccupied, and let $s$ be the number of occupied spaces.

\begin{definition}%
For $r,s,k\in\mathbb N_0$ let $a(r,s,k)$ denote the number of choices for which $r$ spaces remain unoccupied, $s$ spaces are occupied in the end, and $k$ people drive home.
\end{definition}%
This obviously means that there are $n = r+s$ spaces in total, and that $m = k+s$ drivers arrive.  Observe that $\cp(n,m,k)=a(n-m+k,m-k,k)$ is the number of assignments for the car parking problem with~$n$ parking spaces,~$m$ visitors, and~$k$ drivers going home. Correspondingly, $\cp(n,n,k)=a(k,n-k,k)$. Thus, finding a solution for $a(r,s,k)$ will yield a solution for the original problem. We extend this definition to all integers by setting $a(r,s,k)=0$ whenever $r$, $s$, or $k$ is smaller than 0. For the newly introduced numbers, we get the following recursive formula:
\begin{lemma}%
\label{thm:recursion}%
For $r,s,k\in\mathbb N_0$, the number of assignments of $s+k$ drivers to $r+s$ spaces, such that $r$ spaces remain empty, $s$ spaces are occupied, and $k$ drivers leave, is recursively defined by
\begin{equation*}%
a(r,s,k)=
\begin{cases}%
\hfill 1 &\text{if } r=s=k=0, \\
a(r-1,s,0) + \sum_{i=0}^{k+1} \binom{s+k}{k+1-i}\cdot a(r,s-1,i)
&\text{if }k=0\text{ and }(r>0\text{ or }s>0),\\
\hfill \sum_{i=0}^{k+1} \binom{s+k}{k+1-i}\cdot a(r,s-1,i) & \text{if }k>0.\\
\end{cases}%
\end{equation*}%
\end{lemma}%
\begin{proof}%
If~$r=s=k=0$, there exists exactly one assignment. Next, let~$k>0$, as in the third case. Since $k>0$ drivers leave in the end, there are at least~$k+1$ drivers that arrive at parking space number~$r+s$, counting the ones that actually chose it, as well as the ones that did not. An assignment of the~$s+k$ drivers to the~$r+s$ parking spaces satisfies this condition, if and only if for some~$i\in\{1,\hdots,k+1\}$ there are~$k+1-i$ drivers that actually choose the last space and~$i$ drivers that arrive at the space, although they have not chosen it. For different values of $i$, the corresponding assignments of the drivers must differ.

There are~$\binom{r+k}{k+1-i}$ ways to choose the~$k+1-i$ drivers that actually pick the last space and $a(r,s-1,i)$ assignments of the remaining~$s-1+i$ drivers to the first~$r+s-1$ spaces, such that exactly~$i$ of them will arrive at the last space. Since these assignments are independent of the choice of $k+1-i$ drivers that choose the last space, the claimed recursion holds.

Finally, let $k=0$ but $r>0$ or $s>0$. If some driver arrives at the last parking space, whether actually choosing it or just taking it due to lack of other spaces, the same recursion as for $k>0$ holds. Otherwise, the last space will be empty and the number of assignments in which this happens is equal to the number of ways all $s$ drivers can be assigned to the first~$r+s-1$ spaces, such that no driver has to leave. For this, there are exactly $a(r-1,s,0)$ ways.
\end{proof}%
We restate the recursion formula from the previous lemma as
\begin{equation*}%
a(r,s,k)=\textbf{1}_{\{r=s=k=0\}}(r,s,k) + \textbf{1}_{\{k=0\}}(k)\cdot a(r-1,s,0)
+ \sum_{i=0}^{k+1} \binom{s+k}{k+1-i}\cdot a(r,s-1,i)\,,
\end{equation*}%
where $\textbf{1}_{\{k=0\}}(k)$ and $\textbf{1}_{\{r=s=k=0\}}(r,s,k)$ are the characteristic functions which are one, if $k=0$ respectively $r=s=k=0$ and zero, otherwise. Dividing both sides by $\binom{s+k}{k}$ yields
\begin{equation}%
\label{eq:a_recursion}
\frac{a(r,s,k)}{\tbinom{s+k}{k}}=\textbf{1}_{\{r=s=k=0\}}(r,s,k)
+ \textbf{1}_{\{k=0\}}(k)\cdot\frac{a(r-1,s,0)}{\tbinom{s+0}{0}}
+ \frac{s}{k+1}\cdot\sum_{i=0}^{k+1}
\binom{k+1}{i}\cdot\frac{a(r,s-1,i)}{\tbinom{s-1+i}{i}}\,.
\end{equation}%
The previous equation suggests to represent $a(r,s,k)$ by a generating function which is ordinary in $u$ and exponential in a combination of $v$ and $t$.
\begin{lemma}%
\label{lem:a_implicit}%
Let~$A$ be the formal power series in the three variables~$u$,~$w$, and~$t$ defined by
\begin{equation*}%
A(u,v,t):=\sum_{r,s,k\ge0}
a(r,s,k)\cdot u^r \tfrac{v^s t^k}{(s+k)!}\,.
\end{equation*}%
Then $A$ is the unique solution of
\begin{equation*}%
0=(\tfrac{v}{t}e^t-1)\cdot A(u,v,t)+(u-\tfrac{v}{t})\cdot A(u,v,0)+1
\end{equation*}%
in the ring of formal power series in $u$, $v$ and $t$.
\end{lemma}%
\begin{proof}%
Multiplying both sides of \eref{eq:a_recursion} by $u^r \tfrac{v^s}{s!}\tfrac{t^k}{k!}$ and summing both sides over the parameters $r$, $s$, and $k$, followed by the usual manipulation, such as index-shifts and factorizing the product of $A(u,v,t)$ and $e^{t}$, shows that the definition of $A$ and the equation stated in this lemma are equivalent.
\end{proof}%

\section{An Explicit Formula}%
\label{sec:genfun}%

We will now proceed to find an explicit formula for the coefficients of the generating function. In general, equations like the equation in \lref{lem:a_implicit} cannot be solved directly, since both $A(u,v,t)$ and $A(u,v,0)$ are unknown. Instead, we resort to using the so-called kernel method~\cite{Flajolet, Prodinger}. Writing the equation in \lref{lem:a_implicit} as
\begin{equation*}%
K(v,t)\cdot A(u,v,t)=(u-\tfrac{v}{t})\cdot A(u,v,0)+1
\end{equation*}%
with the \emph{kernel} $K(v,t)=1-\tfrac{v}{t}e^t$, we solve  for $A(u,v,0)$ by setting the kernel equal to zero, which is here equivalent to finding a formal power series $t(v)$ for which $K(v,t(v))=0$. The solution to $t=ve^t$ is the well-known tree function $t=T(v)$ enumerating rooted trees on $i$ labelled nodes, which is standardly expressed in terms of the Lambert W-function~\cite{Corless} as $T(v)=-W(-v)$ and has series expansion
\begin{equation*}%
T(v)=\sum_{i=1}^\infty\frac{i^{i-1}}{i!}\cdot v^i\,.
\end{equation*}%

Therefore
\begin{equation}%
\label{eq:genfunctcomplete}%
A(u,v,0)=\frac{e^{T(v)}}{1-ue^{T(v)}}\,,
\end{equation}%
which can be substituted into the equation of \lref{lem:a_implicit} to derive an explicit expression for $A(u,v,t)$. This proves the following lemma.

\begin{lemma}%
\label{lem:genfunctdefective}%
The generating function for the car parking problem is given by
\begin{equation*}%
A(u,v,t)=\frac1{1-\frac vte^t}+\frac{u-\frac vt}{1-\frac vte^t}\cdot \frac{e^{T(v)}}{1-ue^{T(v)}}\,.
\end{equation*}%
\end{lemma}%

Applying Lagrange inversion to \eref{eq:genfunctcomplete}, we obtain the explicit expression
\begin{equation}%
\label{eq:a_null}%
A(u,v,0)=\sum_{r,s\ge 0}(r+1)\cdot (r+s+1)^{s-1}\cdot u^r\tfrac{v^s}{s!}\,.
\end{equation}%

The coefficients of $A(u,v,0)$ have already been obtained in the introduction by a direct combinatorial method. It becomes apparent that we can express the car parking numbers $\cp(n,m,k)=a(n-m+k,m-k,k)$ in terms of the following sums, which also have a direct combinatorial interpretation.

\begin{definition}%
\label{def:s_sum}%
For $n,m,k\in\mathbb N_0$, let
\begin{equation*}%
S(n,m,k) :=%
\begin{cases}%
n^m & \text{if }k\leq m-n,\\
\displaystyle\sum_{i=0}^{m-k} \tbinom{m}{i}\cdot(n{-}m{+}k)\cdot(n{-}m{+}k{+}i)^{i-1}\cdot(m{-}k{-}i)^{m-i}
&\text{otherwise.}%
\end{cases}%
\end{equation*}%
\end{definition}%

The car parking numbers then calculate as follows.

\begin{theorem}%
\label{mainresult}
Let $n,m,k\in\mathbb N_0$. Then, the sum $S(n,m,k)$ counts the number of car parking assignments of $m$ cars on $n$ spaces, such that at least $k$ cars do not find a parking space, that is,
\begin{equation*}%
S(n,m,k)=\sum_{j=k}^m\cp(n,m,j)\,.
\end{equation*}%
Equivalently, the car parking numbers $\cp(n,m,k)$ are given by
\begin{equation*}%
\cp(n,m,k)=S(n,m,k)-S(n,m,k+1)\,.
\end{equation*}%
\end{theorem}%

\begin{proof}%
Expanding the explicit form of $A(u,v,t)$ in \lref{lem:genfunctdefective} with the help of \eref{eq:a_null} leads, after some lengthy manipulations, to
\begin{align*}
A(u,v,t)=&
\sum_{s\ge0}\sum_{k\ge0}\tfrac{s^{s+k}\,v^s\,t^k}{(s+k)!}
+\sum_{r\ge1}\sum_{s\ge0}\sum_{k\ge0}\sum_{i=0}^s
\tbinom{s+k}{i}\cdot r\cdot(r+i)^{i-1}\cdot(s-i)^{s+k-i}\cdot
u^r\tfrac{v^s t^k}{(s+k)!}\\
&-\sum_{r\ge0}\sum_{s\ge1}\sum_{k\ge0}\sum_{i=0}^{s-1}
\tbinom{s+k}{i}\cdot(r+1)\cdot(r+1+i)^{i-1}\cdot(s-1-i)^{s+k-i}\cdot
u^r\tfrac{v^s t^k}{(s+k)!}\\
=&\sum_{r,s,k\ge0}\big(S(r+s,s+k,k)-S(r+s,s+k,k+1)\big)\cdot
\,u^r\tfrac{v^s t^k}{(s+k)!}\,.
\end{align*}
From this, we read off directly that
\begin{equation*}%
a(r,s,k)=S(r+s,s+k,k)-S(r+s,s+k,k+1)\,.
\end{equation*}%
The statement of the theorem follows from the relation between $\cp(n,m,k)$ and $a(r,s,k)$.
\end{proof}%

Note that we can rewrite $S(n,m,k)$ for $k>m-n$ as
\begin{equation*}%
S(n,m,k)=\sum_{i=0}^{m-k}\binom{m}{i}\cp(n-m+k+i-1,i,0)\cdot(m-k-i)^{m-i}\;.
\end{equation*}%
This observation \cite{difrancesco} leads to a direct combinatorial proof of Theorem \ref{mainresult}.
\begin{proof}[Alternative proof]
If at least $k$ cars do not find a parking space, then there are at least $\ell=n+k-m$ empty parking spaces. Assuming the $\ell$-th empty space occurs at position $\ell+i$, there are $i$ cars successfully parked in the $\ell+i-1$ spaces to the left of it, which is counted by $\cp(\ell+i-1,i,0)$. Selecting these $i$ cars out of all $m$ cars can be done in $\binom{m}{i}$ different ways. The remaining $m-i$ cars are assigned to the $n-\ell-i$ rightmost spaces in $(n-\ell-i)^{m-i}$ different ways. Summing over all possible values of $i$ leads to
\begin{equation*}%
S(n,m,k)=\sum_{i=0}^{n-\ell}\binom{m}{i}\cdot\cp(\ell+i-1,i,0)\cdot(n-\ell-i)^{m-i}\;.
\end{equation*}%
\end{proof}%

Another way to derive this theorem is to reinterpret defective parking
functions in terms of $x$--parking functions as described in the
introduction and then apply the results from \cite{pitmanstanley98}.

\section{Abel's Binomial Identity}%
\label{sec:abel}%

In the last section we saw that the number of assignments of $m$ cars to $n$ spaces, such that at least $k$ drivers fail to park, is the sum $S(n,m,k)$. Interestingly, $S(n,m,k)$ turns out to be a partial Abel-type sum as it appears in Abel's Binomial identity \cite{abelbinom}.

\begin{lemma}[Abel's Binomial Identity]%
\begin{equation*}%
\sum_{i = 0}^{m} \binom{m}{i}\cdot a\cdot(a+i)^{i-1}\cdot(b-i)^{m-i} =(a+b)^m\quad \text{ for all } a,b\in \mathbb R, m\in\mathbb N_0 .
\end{equation*}%
\end{lemma}%
In fact, our approach gives a proof of this identity for the case $a,b,m\in \mathbb N_0 $ and $b = m$. (Put $k=0$, $a=n-m$ and $b=m$ in the defining equation of $S(n,m,k)$ in \dref{def:s_sum}). We use this identity to find the following short expressions of $S(n,m,k)$ and $S(n,n,k)$.
\begin{align*}%
S(n,m,k)\,=\,&\begin{cases}
n^m & \text{if } k\leq m{-}n,\\
n^m-\displaystyle\sum_{i=0}^{k-1} \binom{m}{i}\cdot(-1)^i\cdot(n{-}m{+}k)\cdot(k{-}i)^{i}\cdot(n{+}k{-}i)^{m-1-i}&\text{otherwise},
\end{cases}\\
S(n,n,k)\,=\,&
n^n-\sum_{i=0}^{k-1} \binom{n}{i}\cdot(-1)^i\cdot k\cdot(k{-}i)^{i}\cdot(n{+}k{-}i)^{n-1-i}\,.
\end{align*}%

We now give $S(n,n,k)$ for values of $k$ close to zero or $n$:
\begin{align*}%
S(n,n,0)\,=\,&n^n\,, \\
S(n,n,1)\,=\,&n^n-(n+1)^{n-1}\,, \\
S(n,n,2)\,=\,&n^n-2\cdot(n+2)^{n-1}+2\cdot n\cdot(n+1)^{n-2}\,, \\
S(n,n,3)\,=\,&n^n - \tfrac{3}{2}\cdot n\cdot (n-1)\cdot(n+1)^{n-3} + 
    6\cdot n\cdot(n+2)^{n-2} - 3\cdot(n+3)^{n-1}\,, \\
\vdots~& \\
S(n,n,n-3)\,=\,& 3^n+n\cdot(n-3)\cdot 2^{n-1}+\tfrac{1}{2}\cdot n\cdot(n-1)^2\cdot(n-3)\,, \\
S(n,n,n-2)\,=\,& 2^n+n\cdot(n-2)\,, \\
S(n,n,n-1)\,=\,& 1\,, \\
S(n,n,n)\,=\,&
\begin{cases}%
1 & \text{if }n=0\,,\\
0 & \text{otherwise}\,.
\end{cases}%
\end{align*}%

It follows that for $m=n+\ell$ with $k>\ell$ and $\ell\in\mathbb Z$
\begin{equation*}%
\lim_{n\to\infty}n\cdot\left(\frac{S(n,m,k)}{n^m}-1\right)=-\,(k-\ell)\cdot\sum_{i=0}^{k-1}\frac{(-1)^i}{i!}\cdot(k-i)^i\cdot
e^{k-i}=: \phi(\ell,k)\,.
\end{equation*}%
(This limit is trivially zero if $k\leq\ell$.) This implies that for $\ell\leq0$ the limit 
\begin{equation*}%
\lim_{n\to\infty}\frac{\cp(n,n+\ell,k)}{\cp(n,n+\ell,0)} = \lim_{n\to\infty}\frac{\phi(\ell,k)-\phi(\ell,k+1)}{\phi(\ell,0)-\phi(\ell,1)}
\end{equation*}%
is finite (for $\ell>0$ the denominator is zero). For example, we find
\begin{equation*}%
\lim_{n\to\infty}\frac{\cp(n,n,1)}{\cp(n,n,0)}=2e-3
\qquad\mbox{and}\qquad
\lim_{n\to\infty}\frac{\cp(n,n,2)}{\cp(n,n,0)}=3e^2-8e+7/2\,.
\end{equation*}%

We conclude that in a random instance, we do not expect the number of drivers having to leave to be bounded by a constant.

Abel's binomial identity also gives us a first bound on $S(n,m,k)$. We obtain it by bounding the binomial coefficient appearing within the sum: Since for $k>m-n$, we know that
\begin{equation*}%
S(n,m,k)=\sum_{i=0}^{m-k}h(n,m,k)
\cdot\binom{m-k}{i}\cdot (n-m+k)\cdot(n-m+k+i)^{i-1}\cdot(m-k-i)^{m-k-i}
\end{equation*}%
with
\begin{equation*}%
h(n,m,k)=\frac{(m-k-i)^k\cdot(m-k-i)!\cdot m!}{(m-i)!\cdot(m-k)!}\leq \frac{m!}{(m-k)!}\,.
\end{equation*}%
It follows that
\begin{equation*}%
S(n,m,k)\le \frac{m!}{(m-k)!}\cdot n^{m-k}\,.
\end{equation*}%

We have seen how the formula for $S(n,m,k)$, derived by Abel's identity, allows us to obtain asymptotic results for fixed values of $k$. In the following section, we analyze the case in which~$k$  grows as a function of $n$. Some of the upcoming results may  be elementarily proven using the previous inequality.

\section{Asymptotics}%
\label{sec:asymptotics}%

For the parking problem with $m$ cars and $n$ spaces, the probability that at least $k$ drivers cannot park their cars is $S(n,m,k)/n^m$ with $S(n,m,k)$ as defined in \dref{def:s_sum}.  The case $k=0$ corresponds to partially filled hash tables and has been analyzed at length in~\cite{chassaing_louchard}, with the most interesting asymptotic behavior obtained for $n-m=O(\sqrt n)$. Similarly, we find non-trivial behavior in the regime where both $n-m$ and $k$ are of order $O(\sqrt n)$.

\begin{theorem}%
\label{thm:asym}%
Let $x\in\mathbb R^+$ and $y\in\mathbb R$. Then, the limiting probability that in a random assignment of $n+\lfloor y\sqrt n\rfloor$ drivers to $n$ spaces at least $\lfloor x\sqrt n\rfloor$ drivers fail to park is
\begin{equation*}%
\lim_{n\to\infty}\frac{S(n,n+\lfloor y\sqrt n\rfloor,\lfloor x\sqrt n\rfloor)}{n^{n+\lfloor y\sqrt n\rfloor}}
=\begin{cases}%
e^{-2x(x-y)} &\text{if } x>y, \\
1 &\text{otherwise.}
\end{cases}
\end{equation*}%
\end{theorem}%

\begin{proof}%
Let $p(n,m,k)=S(n,m,k)/n^m$. First note that the case $x\leq y$ corresponds to the case $m\geq n+k$, where $p(n,m,k)=1$. The case $x>y$ corresponds to the case $m<n+k$, where
\begin{equation*}%
p(n,m,k)=\sum_{i=0}^{m-k}p(n,m,k,i)
\end{equation*}%
with $p(n,m,k,i)$ given by
\begin{equation*}%
p(n,m,k,i)=\frac{1}{n^m}\cdot\binom{m}{i}\cdot(m-k-i)^{m-i}\cdot(n-m+k+i)^{i-1}\cdot(n-m+k)\,.
\end{equation*}%
A straightforward but slightly tedious calculation establishes that for $\alpha\in[0,1]$
\begin{equation*}%
\lim_{n\to\infty}n\,p(n,n+y\sqrt n,x\sqrt n,\alpha n)
=\frac{x-y}{\sqrt{2\pi\alpha^3(1-\alpha)}}\cdot\exp\left(-\frac{(x-(1-\alpha)y)^2}{2\alpha(1-\alpha)}\right)\,.
\end{equation*}%

As we have uniform convergence to a bounded limiting function for $\alpha\in[0,1]$, it is permissible to approximate $p(n,n+\lfloor y\sqrt n\rfloor,\lfloor x\sqrt n\rfloor)$ for large $n$ by an integral as follows:
\begin{align*}%
\lim_{n\to\infty}p(n,n+\lfloor y\sqrt n\rfloor,\lfloor x\sqrt n\rfloor)
&=\lim_{n\to\infty}
\sum_{i=0}^{n-\lfloor x\sqrt n\rfloor}p(n,n+\lfloor y\sqrt n\rfloor,\lfloor x\sqrt n\rfloor,i)\\
&=\lim_{n\to\infty}\int_0^{1-x/\sqrt n}n\cdot p(n,n+y\sqrt n,x\sqrt n,\alpha n)\,d\alpha\\
&=\int_0^1\frac{x-y}{\sqrt{2\pi\alpha^3(1-\alpha)}}\cdot\exp\left(-\frac{(x-(1-\alpha)y)^2}{2\alpha(1-\alpha)}\right)d\alpha\,.
\end{align*}%

Under the substitution $\alpha=\frac{u(x-y)}{x+u(x-y)}$, this integral simplifies to
\begin{equation*}%
\frac1{\sqrt{2\pi}}\cdot\int_0^\infty
\sqrt{\frac{x(x-y)}{u^3}}
\cdot\exp\left(-\,x\cdot(x-y)\cdot\frac{(1+u)^2}{2u}\right)\,du
=\exp(-2x(x-y))\,.
\end{equation*}%
\end{proof}%

This theorem implies that for $m<n+k$, a good approximation is given by
\begin{equation*}%
\frac{\cp(n,m,k)}{n^m}\approx\frac{2}{n}\cdot(2k-m+n)
\cdot e^{-2k(k-m+n)/n}\,.
\end{equation*}%
\fref{fig:distribution1} shows a comparison between $\cp(n,m,k)/n^m$ and this approximation for the three qualitatively different scenarios $m<n$, $m=n$ and $m>n$.

\begin{figure}%
\begin{center}%
\includegraphics[width=0.32\textwidth,angle=0]{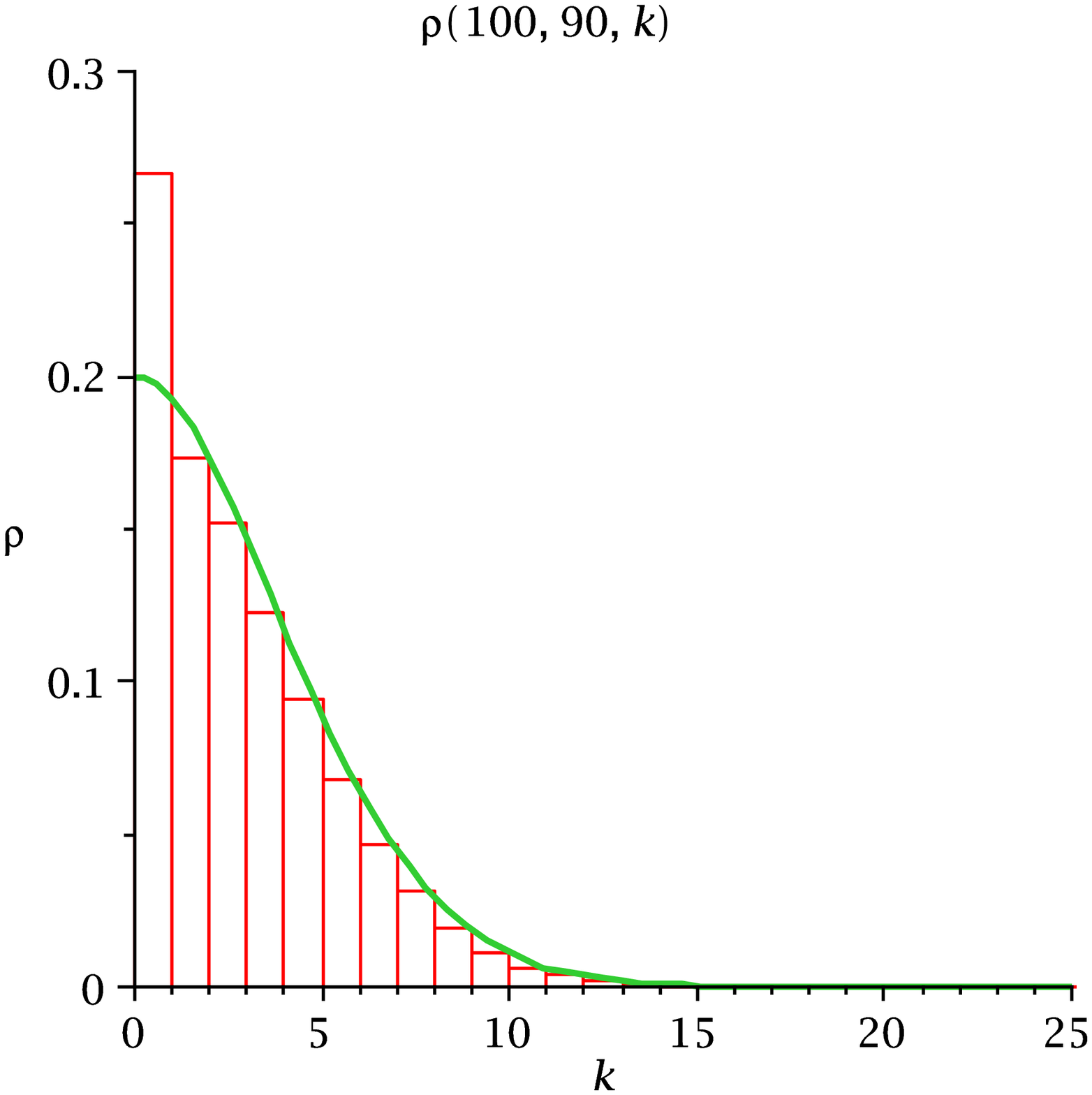}%
\includegraphics[width=0.32\textwidth,angle=0]{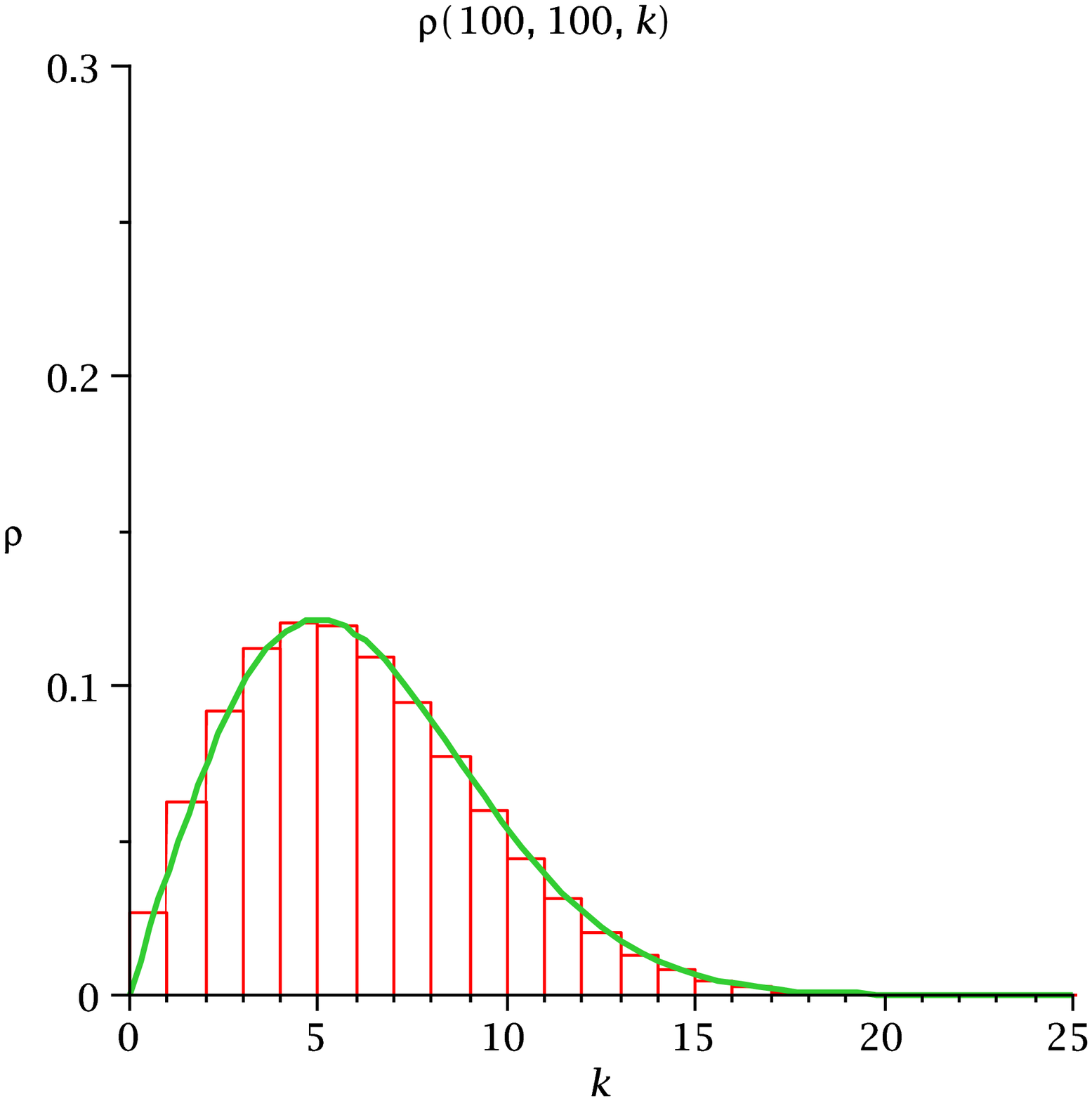}%
\includegraphics[width=0.32\textwidth,angle=0]{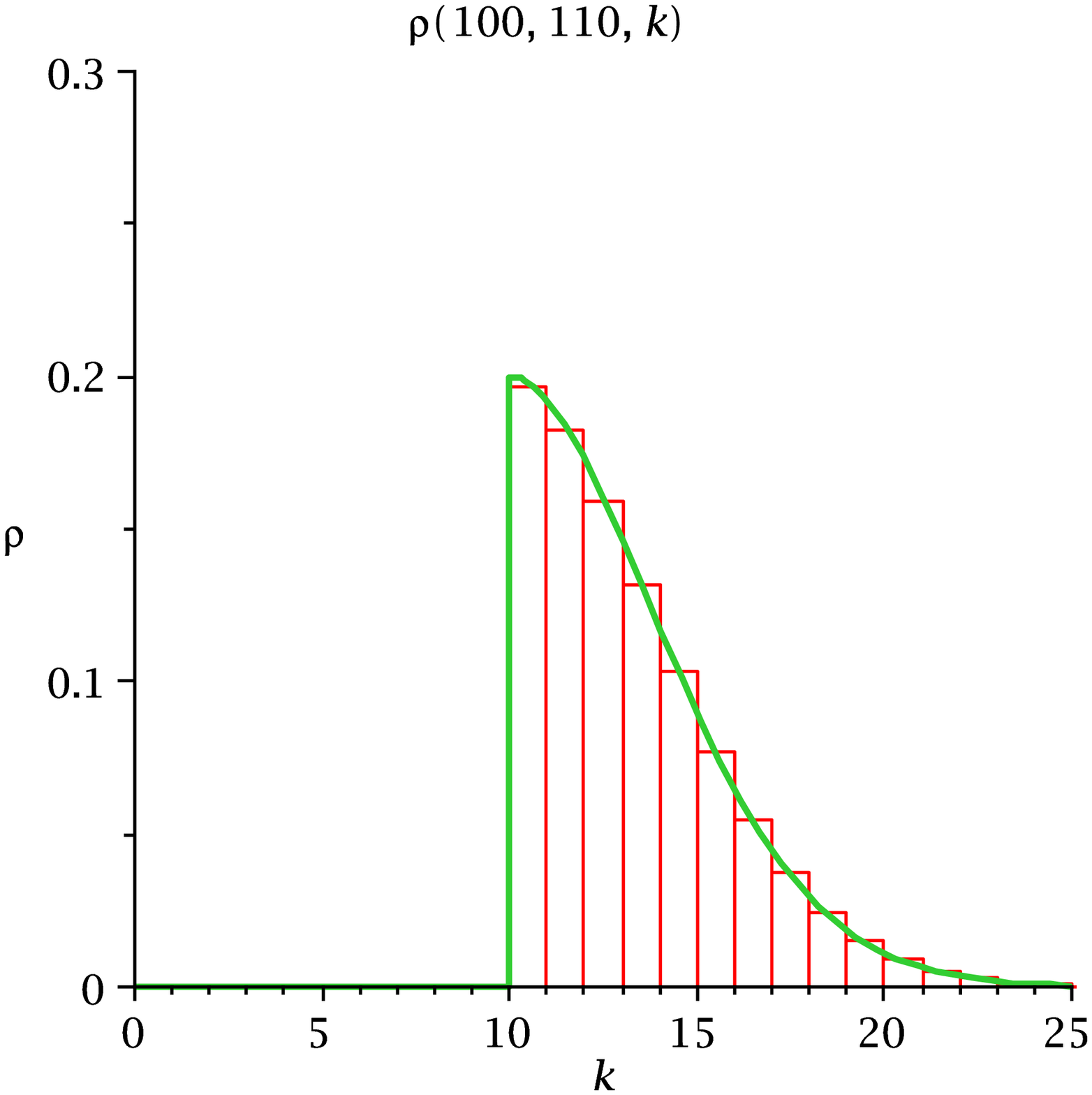}%
\end{center}%
\caption{\label{fig:distribution1}
This figure shows a comparison between the probabilities $\rho(n,m,k)=\cp(n,m,k)/n^m$ that $m$ cars park randomly on $n$ spaces, such that $k$ drivers fail to park, and their asymptotic approximation for $n=100$ and $m=90$ (left), $m=100$ (middle) and $m=110$ (right), respectively.}
\end{figure}%

\tref{thm:asym} also shows that in the special case $m=n$, where the number of drivers and the number of parking spaces coincide, a random assignment will result in $\sqrt n$ drivers leaving. In particular we get the following corollary:
 
\begin{corollary}%
\label{cor:root}%
Let $k\colon\mathbb N\to\mathbb N$. Then, the limiting probability that in a random assignment of $n$ drivers to $n$ spaces at least $k(n)$ drivers fail to park is
\begin{equation*}%
\lim_{n \rightarrow \infty }\frac{S(n,n,k(n))}{n^n}=
\begin{cases}%
0 & \text{if } \lim_{n\to\infty} k(n)/\sqrt n=\infty, \\
1 & \text{if } \lim_{n\to\infty}k(n)/\sqrt n=0.
\end{cases}%
\end{equation*}%
\end{corollary}%

Another question one may wish to ask is how the number $m$ of cars needs to scale with the number $n$ of parking spaces to fill the car park with a finite limiting probability, and when this probability reaches one. Similarly, from the viewpoint of the coupon collector, it is reasonable to ask how many coupons are needed to obtain a complete set.

Recall that the quantity $S(n,m,m-n+1)$ counts the number of car parking assignments of $m$ cars on $n$ spaces, such that at least $m-n+1$ cars do not find a parking space, or, equivalently, such that at most $n-1$ cars do find a parking space. Therefore, the probability that the car park is full is given by
\begin{equation*}%
\frac{\cp(n,m,m-n)}{n^m}=1-\frac{S(n,m,m-n+1)}{n^m}\,.
\end{equation*}%
We find non-trivial behavior when $m$ depends linearly on $n$ as stated in the following theorem, a similar version of which can be found in \cite{SpencerYan}.

\begin{theorem}%
Let $\lambda\in\mathbb R^+$.
Then, the limiting probability that in a random assignment of $\lfloor \lambda n\rfloor$ drivers to $n$ spaces
all spaces are occupied is
\begin{equation*}%
\lim_{n\to\infty}\frac{\cp(n,\lfloor\lambda n\rfloor,\lfloor\lambda n\rfloor-n)}{n^{\lfloor\lambda n\rfloor}}=
\begin{cases}%
0&\text{if }\lambda\leq1,\\
1-\frac{1}{\lambda}\cdot T\left(\lambda e^{-\lambda}\right)&\text{if }\lambda>1.
\end{cases}
\end{equation*}%
\end{theorem}%

\begin{proof}
Let $p(n,m)=S(n,m,m-n+1)/n^m$. We have
\begin{equation*}%
p(n,m)= \frac{1}{n^m}\cdot\sum_{i=0}^{n-1}\binom{m}{i}\cdot(i+1)^{i-1}\cdot(n-1-i)^{m-i}\,.
\end{equation*}%
Choosing $m=\lfloor\lambda n\rfloor$, we exchange summation with taking the limit and compute the limit term-wise, arriving at
\begin{equation}%
\label{eq:sumasy}
\lim_{n\to\infty}p(n,\lfloor\lambda n\rfloor)=
\sum_{i=0}^\infty\frac{\lambda^i}{i!}\cdot(i+1)^{i-1}\cdot e^{-\lambda(1+i)}\,.
\end{equation}%
Exchanging summation with taking the limit is justified by the Lebesgue dominated convergence theorem, as each term is bounded by the limiting expression, that is,
\begin{equation*}%
0\leq
\frac{1}{n^m}\cdot\binom{m}{i}\cdot(i+1)^{i-1}\cdot(n-1-i)_+^{m-i}\leq\frac{(m/n)^i}{i!}\cdot(i+1)^{i-1}\cdot
e^{-(1+i)m/n}\,.
\end{equation*}%

The sum in \eref{eq:sumasy} converges for $|\lambda e^{-\lambda}|\leq1$ (in particular for all positive $\lambda$). It can be expressed using the tree function $T(v)$ with $v=\lambda e^{-\lambda}$, and we find
\begin{equation*}%
\lim_{n\to\infty}p(n,\lfloor\lambda n\rfloor)=\frac1\lambda\cdot T\left(\lambda e^{-\lambda}\right)\,.
\end{equation*}%

Recalling that $t=T(v)$ solves $v=te^{-t}$, we find that determining $t=T\left(\lambda e^{-\lambda}\right)$ reduces to solving $te^{-t}=\lambda e^{-\lambda}$. For $\lambda\leq1$, we find $t=\lambda$, however, for $\lambda>1$ no such simplification is possible.
\end{proof}%

\begin{figure}%
\begin{center}%
\includegraphics[width=0.32\textwidth,angle=0]{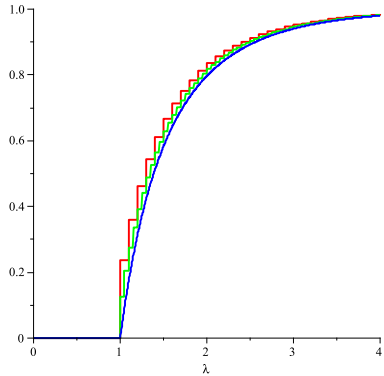}%
\end{center}%
\caption{\label{fig:distribution2}
This chart shows a comparison of the probabilities $\cp(n,m,m-n)/n^m$ for $n=10$, $n=20$, and the limiting curve given by $1-T(\lambda e^{-\lambda})/\lambda$ from top to bottom, respectively. Here, $m=\lfloor\lambda n\rfloor$.}
\end{figure}%

\fref{fig:distribution2} shows a comparison between $\cp(n,m,m-n)/n^m$ with $m=\lfloor\lambda n\rfloor$ and finite values of $n$ and the limiting curve given by $1-T(\lambda e^{-\lambda})/\lambda$.

We conclude with a corollary of the previous theorem which emphasizes the threshold character of its statement. In this sense, the relation between the following corollary and the previous theorem reflects the relation between \cref{cor:root} and \tref{thm:asym}.

Since $\frac{1}{e^{\lambda-1}-\lambda}$ is an upper bound for $T(\lambda e^{-\lambda})/\lambda$ which converges to zero for $\lambda\to\infty$, almost all random assignments of $m=m(n)$ drivers to $n$ parking spaces will result in all spaces being occupied if $m/n$ tends to infinity.

\begin{corollary}%
Let $m\colon\mathbb N\to\mathbb N$. Then, the limiting probability that in a random assignment of $m(n)$ drivers to $n$ spaces all spaces are occupied is
\begin{equation*}%
\lim_{n\to\infty}\frac{\cp(n,m(n),m(n)-n)}{n^{m(n)}}=
\begin{cases}%
0 & \text{if } \lim_{n\to\infty} m(n)/n\le 1,\\
1 & \text{if } \lim_{n\to\infty} m(n)/n=\infty.
\end{cases}%
\end{equation*}%
\end{corollary}%

\section{Conclusion}%
\label{sec:conclusion}%

We have derived the generating function for the defective car parking numbers $\cp(n,m,k)$ and have used it to solve the problem of their exact and asymptotic enumeration. They are closely  connected to the $x$--parking functions for which numerous structural interpretations exist, and these interpretations therefore transfer.  However, we expect applications of our results in different areas, since parking functions, with or without defect, naturally capture various time-dependent models.

For example, motivated by hashing with linear probing, a drop-push model for percolation was proposed in \cite{majumdar_dean}. Here, particles are dropped sequentially on a substrate, followed by the transport of the dropped particles via a pushing mechanism, caused by a local hard-core repulsion between particles on the substrate. If the transport is unidirectional, this is identical to the parking problem studied in this paper.

Another example, the Abelian sandpile model, allows a decrease in the quantity of the system that corresponds to the number of cars. Such a decrease is not only possible but even necessary to prove certain stability properties of the system. Translated to parking functions, the notion of a defect naturally appears~\cite{dhar00}.

A further potential field of application is queueing theory. For instance, the  branching process described in  \cite{DumitriuSpencerYan03} can by viewed in this context. With respect to queueing, parking spaces are interpreted as time slots, at each of which exactly one task (represented by one car) can be processed. Whereas in a branching process, at least one unprocessed task must exist at every point in time; there is no reason to forbid idle time in a more general queueing process.

Finally, we would like to point out that Alois Panholzer has independently obtained related results  \cite{panholzer}, notably a detailed description of limiting distributions.

\section{Acknowledgment}%

 This paper was written during the authors' visit to the programme on Combinatorics and Statistical Mechanics at the Isaac Newton Institute in Cambridge. We are grateful to the Institute for its inspiring atmosphere and hospitality. We would also like to express our gratitude to Philippe Di Francesco for pointing out to us the observation that led to the alternative proof to \tref{mainresult}.

\bibliographystyle{abbrv}%

\end{document}